\documentclass[12pt]{article}
\usepackage{fullpage}
\usepackage{amssymb, amsmath, amsthm, enumerate, math tools}

\DeclarePairedDelimiter\floor{\lfloor}{\rfloor}

\theoremstyle{plain}
\newtheorem{thm}{Theorem}
\theoremstyle{definition}

\newtheorem{lemma}[thm]{Lemma}

\begin{document}
\title{On the number of partitions of a number into distinct divisors}
\author{Noah Lebowitz-Lockard and Joseph Vandehey \\
Department of Mathematics \\
University of Texas at Tyler, Tyler, TX 75799 \\
nlebowitzlockard@uttyler.edu \\
jvandehey@uttyler.edu}
\maketitle

\begin{abstract} Let $p_{\textrm{dsd}} (n)$ be the number of partitions of $n$ into distinct squarefree divisors of $n$. In this note, we find a lower bound for $p_{\textrm{dsd}} (n)$, as well as a sequence of $n$ for which $p_{\textrm{dsd}} (n)$ is unusually large.
\end{abstract}

\newline


\section{Introduction}

A \emph{partition} of $n$ is a representation of $n$ as an unordered sum of positive integers. We let $p(n)$ and $p_d (n)$ be the number of partitions of $n$ and the number of partitions of $n$ into distinct parts. In 1918, Hardy and Ramanujan \cite[\S $1.4$, $7.1$]{HR} proved two of the seminal results on partitions, obtaining asymptotic formulae for $p(n)$ and $p_d (n)$.

\begin{thm} As $n \to \infty$, we have
\[p(n) \sim \frac{1}{4n\sqrt{3}} \exp\left(\pi \sqrt{\frac{2n}{3}}\right), \quad p_d (n) \sim \frac{1}{4\sqrt[4]{3n^3}} \exp\left(\pi \sqrt{\frac{n}{3}}\right).\]
\end{thm}

Seven decades later, Bowman submitted a problem to the American Mathematical Monthly, asking for an asymptotic formula for the number of partitions of $n$ into \emph{divisors} of $n$. Erd{\H o}s and Odlyzko then found precise bounds for this quantity, which we call $p_{\textrm{div}} (n)$ \cite[Seq. A018818]{OE}. From here on, we also let $d(n)$ be the number of divisors of $n$.

\begin{thm}[{\cite{BEO}}] \label{main thm} As $n \to \infty$, we have
\[n^{(1 + O(1/\log \log n))(d(n)/2 - 1)} \leq p_{\textrm{div}} (n) \leq n^{(1 + o(1)) d(n)/2}.\]
\end{thm}

In this note, we consider another type of partition. We let $p_{\textrm{dd}} (n)$ and $p_{\textrm{dsd}} (n)$ be the number of partitions of $n$ into distinct divisors and distinct \emph{squarefree} divisors of $n$, respectively. Though these quantities appear in the Online Encyclopedia of Integer Sequences \cite[Seqs. A033630, A225245]{OE}, we are unaware of any published research on them.

Pseudoperfect numbers are numbers which can be expressed as sums of distinct proper divisors, i.e., solutions to $p_{\textrm{dd}} (n) > 1$. These numbers have a rich history \cite[\S B2]{G}. Erd{\H o}s \cite[Thm. $2$]{E} notably showed that the pseudoperfect numbers have a positive density less than $1$.

The functions $p_{\textrm{dd}} (n)$ and $p_{\textrm{dsd}} (n)$ have very erratic behavior. Let $\sigma(n)$ be the sum of the divisors of $n$. We say that $n$ is \emph{abundant} if $\sigma(n) > 2n$, \emph{deficient} if $\sigma(n) < 2n$, and \emph{perfect} if $\sigma(n) = 2n$. If $n$ is deficient, then $p_{\textrm{dd}} (n) = 1$ and if $n$ is perfect, then $p_{\textrm{dd}} (n) = 2$. While the perfect numbers are sparse, the deficient numbers have a density of approximately $0.7524$ \cite{K}, which implies that $p_{\textrm{dd}} (n) = 1$ about $3/4$ of the time. Even still, we show that $p_{\textrm{dd}} (n)$ and $p_{\textrm{dsd}} (n)$ can be quite large.

\begin{thm} \label{main thm} For a given $i$, let $p_i$ be the $i$th prime. If $n = p_1 p_2 \cdots p_k$ for some $k$, then
\[p_{\textrm{dd}} (n) = p_{\textrm{dsd}} (n) \gtrsim \frac{2^{d(n)/4} \log n}{n \log \log n} = \exp\left(\exp\left((\log 2 + o(1)) \frac{\log n}{\log \log n}\right)\right).\]
\end{thm}

This result is very close to being optimal in the sense that $p_{\textrm{dd}} (n)$ cannot be substantially larger than the bound in Theorem \ref{main thm}. Because $n$ has $d(n)$ divisors, there are $2^{d(n)}$ sets of divisors of $n$, which implies that $p_{\textrm{dd}} (n) \leq 2^{d(n)}$. In addition \cite[Thm. $317$]{HR},
\[d(n) \leq 2^{(1 + o(1)) \log n/\log \log n},\]
which implies that
\[p_{\textrm{dd}} (n) \leq \exp\left(\exp\left((\log 2 + o(1)) \frac{\log n}{\log \log n}\right)\right)\]
for all $n$.

We can actually get a slightly better upper bound than $2^{d(n)}$. A classic theorem of S{\' a}rk{\" o}zy and Szemer{\' e}di \cite{SSz} states that for a given $m, k$, a set of $k$ real numbers has at most
\[(1 + o(1)) \frac{8}{\sqrt{\pi}} \frac{2^k}{k^{3/2}}\]
subsets with sum $m$ and that this bound is optimal. (Recent results such as \cite[Thm. $2.1$]{NV} suggest that one can get a better bound if the elements of the set do not lie in a small number of arithmetic progressions. It would be interesting to know if this condition applies to the divisors of a given $n$.) Hence,
\[p_{\textrm{dd}} (n) \lesssim \frac{8}{\sqrt{\pi}} \frac{2^{d(n)}}{d(n)^{3/2}}.\]

We also find some additional lower bounds on $p_{\textrm{dd}} (n)$ and $p_{\textrm{dsd}} (n)$ which may be of independent interest. In particular, if $n$ is a multiple of a large power of $2$, we can obtain the following result.

\begin{thm} \label{penultimate} Let $n = 2^a m$ with $m > 1$ odd. If $2^{a + 1} > \sigma(m)$, then
\[p_{\textrm{dd}} (n) \geq \floor*{\frac{2^{a + 1} - 1}{\sigma(m) - 1}}^{d(m) - 1}.\]
\end{thm}

As $m \to \infty$, we have $\sigma(m) = m^{1 + o(1)}$ (which we discuss in more detail later.) Fix $\epsilon > 0$. If $n = 2^a m$ is a sufficiently large number and $m < n^{(1/2) - \epsilon}$, then $2^{a + 1} > \sigma(m)$, allowing us to use the previous theorem. Thus, then number of $n \leq x$ satisfying the conditions of Theorem \ref{penultimate} is equal to $x^{(1/2) + o(1)}$.

A similar result holds for squarefree divisors. From here on, $\textrm{rad} (m)$ is the radical of $m$, i.e., the largest squarefree divisor of $m$.

\begin{thm} \label{ultimate} Let $n = q_1 q_2 \cdots q_k m$ where the $q_i$'s are an increasing sequence of primes with $q_1 = 2$ and $q_{i + 1} \leq \sigma(q_1 \cdots q_i) + 1$. If $q_1, \ldots, q_k \nmid m$ and
\[\sigma(q_1 \cdots q_k) < n < \sigma(q_1 \cdots q_k) (\sigma(\textrm{rad} (m)) - 1),\]
then
\[p_{\textrm{dsd}} (n) \geq \floor*{\frac{\sigma(q_1 \cdots q_k)}{\sigma(\textrm{rad}(m)) - 1}}^{d(\textrm{rad}(m)) - 1}.\]
\end{thm}

Note that if $p_i$ is the $i$th prime, then Euclid's proof of the infinitude of primes implies that $p_{i + 1} \leq p_1 \cdots p_i + 1$. So, our bound applies to $n = p_1 \cdots p_k m$, where $m$ is a number whose prime factors are greater than $p_k$.

\section{The main result}

We begin this section by showing that for certain numbers $m$, we can write all numbers up to a given bound as a sum of distinct squarefree divisors of $m$.

\begin{lemma} \label{lemma 2} Let $m = q_1 q_2 \cdots q_k$, where the $q_i$'s are distinct primes, $q_1 = 2$, and $q_{i + 1} \leq \sigma(q_1 q_2 \cdots q_i) + 1$ for all $i < k$. Then we can express every number $\leq \sigma(m)$ as a sum of distinct divisors of $m$.
\end{lemma}

\begin{proof} We prove this result by induction on $k$. Clearly, it holds for $k = 1$ because we can express $1$, $2$, and $3$ as sums of distinct divisors of $2$. Suppose $k > 1$ and we already have the result for $k - 1$. For a given $q_k \leq (q_1 + 1) \cdots (q_{k - 1} + 1) + 1$, we have that
\[\bigcup_{i, j \leq (q_1 + 1) \cdots (q_{k - 1} + 1)} \{iq_k + j\} = [0, (q_1 + 1) \cdots (q_k + 1)].\]
By assumption, we can express every number $\leq (q_1 + 1) \cdots (q_{k - 1} + 1)$ as a sum of distinct divisors of $q_1 \cdots q_{k - 1}$. In particular, for any $i, j \leq (q_1 + 1) \cdots (q_{k - 1} + 1)$, we can write
\[i = \sum_{d \in S_1} d, \quad j = \sum_{d \in S_2} d\]
for some sets $S_1$ and $S_2$ of divisors of $q_1 \cdots q_{k - 1}$. So,
\[iq_k + j = \sum_{d \in S_1} dq_k + \sum_{d \in S_2} d,\]
which is a sum of distinct divisors of $m$.
\end{proof}

Using this result, we prove our main theorem.

\begin{proof}[Proof of Theorem \ref{main thm}] Let $n = p_1 \cdots p_k$ and let $C > e^\gamma - \epsilon$ for a small positive $\epsilon$, where $\gamma$ is the Euler-Mascheroni constant. In addition, let $q$ be the prime closest to $C \log \log n$. Because the ratio of consecutive primes goes to $1$, we have $q \sim C \log \log n$ as $n \to \infty$. If $n$ is sufficiently large, then $q < p_k \sim \log n/\log \log n$, and so we have $q | n$. From here on, we let $n = qp_k m$.

There are $2^{d(m)}$ sets of distinct divisors of $m$, each of which has a sum of at most $\sigma(m)$. By the Pigeonhole Principle, there exists some $a \leq \sigma(m)$ which has at least $2^{d(m)}/\sigma(m)$ representations as a sum of distinct divisors of $m$. In addition, $\sigma(m) - a$ also has at least $2^{d(m)}/\sigma(m)$ representations because we can simply take the complement of any subset of the set of divisors of $m$ which add up to $a$. If we let $A = \max(a, \sigma(m) - a)$, we have a number $\geq \sigma(m)/2$ which we can write as a sum of distinct divisors of $m$ in at least $2^{d(m)}/\sigma(m)$ different ways.

At this point, we show that the set $\{p_1, p_2, \ldots, p_k\} \backslash \{q\}$ satisfies the conditions of the previous lemma. If $p_r < q$, then Euclid's proof of the infinitude of the primes shows that $p_r \leq p_1 \cdots p_{r - 1} + 1$. Suppose $p_r > q$ with $r \leq k$. Then the product of the primes $< p_r$ excluding $q$ is still asymptotic to $e^{p_r}$, which is much larger than $p_r$.

Let $n = qA + B$. We already know that we can express $A$ as a sum of distinct divisors of $m$ in at least $2^{d(m)}/\sigma(m)$ ways. If we can show that $B \geq 0$ and that it is possible to express $B$ as a sum of distinct divisors of $p_k m$, then we will be able to find at least $2^{d(m)}/\sigma(m)$ expressions for $n$ as a sum of distinct divisors of $n$. Simply take each sum for $A$ and multiply every element by $q$, then add the sum for $B$.

We prove that $B \in [0, \sigma(p_k m)]$. From there, Lemma \ref{lemma 2} implies that $B$ is a sum of distinct divisors of $p_k m$. In order to prove that $B \geq 0$, we need to show that $qA \leq n$. Note that $qA \leq q\sigma(m)$. Mertens' Theorem \cite[Thm. $429$]{HW} gives us
\[q \sigma(m) = qm \prod_{\substack{p \leq p_{k - 1} \\ p \neq q}} \left(1 + \frac{1}{p}\right) \sim e^\gamma \frac{n}{p_k} \log p_{k - 1} \sim e^\gamma \frac{n (\log \log n)^2}{\log n}.\]
If $n$ is sufficiently large, then $qA < n$.

We now show that $B \leq \sigma(p_k m)$. Because $A$ is positive, $B = n - qA < n$. We prove that $B \leq \sigma(p_k m)$ by showing that $\sigma(p_k m) > n$. We apply Mertens' Theorem again, obtaining
\[\sigma(p_k m) \sim e^\gamma p_k m \log \log (p_k m) \sim n(e^\gamma \log \log n)/q \sim (e^\gamma/(e^\gamma - \epsilon)) n.\]

Putting everything together gives us $p_{\textrm{dsd}} (n) \geq 2^{d(m)}/\sigma(m)$. In addition,
\[\sigma(m) \sim e^\gamma m \log \log m \sim e^\gamma \frac{n \log \log n}{qp_k} \sim \frac{e^\gamma}{e^\gamma - \epsilon} \frac{n \log \log n}{\log n}.\]
We also have $d(m) = d(n)/4$. Letting $\epsilon \to 0$ gives us our desired result.
\end{proof}

\section{Theorems \ref{penultimate} and \ref{ultimate}}

In order to prove Theorems \ref{penultimate} and \ref{ultimate}, we provide alternate characterizations of $p_{\textrm{dd}} (n)$ and $p_{\textrm{dsd}} (n)$ in terms of lattice points. From here on, we let $\mathcal{D} (k)$ be the set of divisors of an integer $k$.

\begin{lemma} \label{2^a m} If $n = 2^a m$ with $m$ odd, then
\[p_{\textrm{dd}} (n) = \#\left\{(x_d)_{d \in \mathcal{D}(m) \backslash \{1\}} : x_d \leq 2^{a + 1} - 1 \textrm{ and } n - 2^{a + 1} + 1 \leq \sum_{d \in \mathcal{D}(m) \backslash \{1\}} dx_d \leq n\right\}.\]
\end{lemma}

\begin{proof} Let $S$ be a set of divisors of $n$ with sum $n$. Every element of $S$ has the form $2^b d$ with $b \leq a$ and $d | m$. Let $x_d$ be the sum of $2^b$ for all $b$ satisfying $2^b d \in S$. Then,
\[n = \sum_{s \in S} s = \sum_{d | m} d \sum_{2^b d \in S} 2^b = \sum_{d | m} dx_d.\]
By definition, $x_d$ can be any sum of distinct powers of $2$ up to $2^a$. Hence, $x_d$ can be any non-negative integer less than $2^{a + 1}$. We have
\[\sum_{d \in \mathcal{D}(m) \backslash \{1\}} x_d d = n - x_1.\]
Setting $d$ to $1$ shows that $x_1$ can be any non-negative integer $\leq 2^{a + 1} - 1$. Therefore,
\[n - 2^{a + 1} + 1 \leq \sum_{d \in \mathcal{D}(m) \backslash \{1\}} x_d d \leq n,\]
where the only restriction on the $x_d$'s is that they are also $\leq 2^{a + 1} - 1$.

Conversely, suppose $(x_d)_{d \in \mathcal{D}(m) \backslash \{1\}}$ is a tuple satisfying the conditions of the theorem. If we let
\[x_1 = n - \sum_{d \in \mathcal{D}(m) \backslash \{1\}} dx_d,\]
then the sum of all $dx_d$ is equal to $n$. We also observe that $x_1 \geq 0$ because the sum of $dx_d$ over all $d \in \mathcal{D}(m) \backslash \{1\}$ is at most $n$. In addition, we can break each $x_d$ into powers of $2$. For each $d \in \mathcal{D}(m)$, let $T_d$ be the unique set of non-negative integers satisfying
\[x_d = \sum_{i \in T_d} 2^i.\]
By assumption, $x_d \leq 2^{a + 1} - 1$ for all $d > 1$. In addition, $x_1 \leq 2^{a + 1} - 1$ because the sum of $dx_d$ over all possible $d > 1$ is at most $n - 2^{a + 1} - 1$. Therefore, every set $T_d$ consists of elements $\leq a$. We now have a representation
\[n = \sum_{d \in \mathcal{D}(m)} x_d \sum_{i \in T_d} 2^i.\]
Each number $2^i x_d$ is a distinct divisor of $n$ and the set of pairs $(2^i, x_d)$ is uniquely determined by $(x_d)_{d \in \mathcal{D}(m) \backslash \{1\}}$. Therefore, each tuple corresponds to a different representation of $n$ as a sum of distinct divisors of $n$.
\end{proof}

\begin{lemma} \label{squarefree lemma} Let $n = q_1 \cdots q_k m$ where $q_1, \ldots, q_k$ is an increasing sequence of primes with $q_1 = 2$ and $q_{i + 1} \leq \sigma(q_1 \cdots q_i) + 1$ for all $i$ and $q_1, \ldots, q_k \nmid m$. In addition, let $\mathcal{S} = \mathcal{D} (\textrm{rad} (m))$ be the set of squarefree divisors of $m$. Then,
\[p_{\textrm{dsd}} (n) = \#\left\{(x_d)_{d \in \mathcal{S} \backslash \{1\}} : x_d \leq \sigma(q_1 \cdots q_k) \textrm{ and } n - \sigma(q_1 \cdots q_k) \leq \sum_{d \in \mathcal{S} \backslash \{1\}} dx_d \leq n\right\}.\]
\end{lemma}

\begin{proof} Our proof is similar to the proof of the previous lemma. The squarefree divisors of $m$ are simply the divisors of $\textrm{rad}(m)$. Every sum of distinct squarefree divisors of $n$ has the form
\[\sum_{d | \textrm{rad} (m)} d \sum_{s \in S_d} s,\]
where $S_d$ is a set of divisors of $q_1 \cdots q_k$.

Using these results, we can bound $p_{\textrm{dsd}} (n)$ from above. Lemma \ref{lemma 2} implies that we can express every number up to $\sigma(q_1 \cdots q_k)$ as a sum of distinct divisors of $q_1 \cdots q_k$. So, we can rewrite any sum of distinct divisors of $n$ in the form
\[x_1 + \sum_{d \in \mathcal{S}} dx_d,\]
with each $x_d \leq \sigma(q_1 \cdots q_k)$. If this quantity equals $n$, the rightmost sum must lie in the interval $[n - \sigma(q_1 \cdots q_k), n]$ because $x_1 \leq \sigma(q_1 \cdots q_k)$. An argument similar to the last paragraph of the previous proof implies that every tuple $(x_d)_{d \in \mathcal{S} \backslash \{1\}}$ corresponds to a unique representation of $n$ as a sum of distinct divisors of $n$.
\end{proof}

\begin{proof}[Proof of Theorem \ref{penultimate}] Let $(x_d)_{d \in \mathcal{D} (m) \backslash \{1\}}$ be a tuple of integers which lie in the interval
\[\left[\frac{n - 2^{a + 1} + 1}{\sigma(m) - 1}, \frac{n}{\sigma(m) - 1}\right].\]
We show that this tuple satisfies the conditions of Lemma \ref{2^a m}.

We note that for each $x_d$, we have
\[x_d \leq \frac{n}{\sigma(m) - 1} = \frac{2^a m}{\sigma(m) - 1} \leq 2^a \leq 2^{a + 1} - 1,\]
which shows that the $x_d$'s are not too large. In addition,
\[\sum_{d \in \mathcal{D} (m) \backslash \{1\}} dx_d \leq \left(\sum_{d \in \mathcal{D} (m) \backslash \{1\}} d\right) \frac{n}{\sigma(m) - 1} = n.\]
For the lower bound on this sum, we note that
\[\sum_{d \in \mathcal{D} (m) \backslash \{1\}} dx_d \geq \left(\sum_{d \in \mathcal{D} (m) \backslash \{1\}} d\right) \frac{n - 2^{a + 1} + 1}{\sigma(m) - 1} = n - 2^{a + 1} - 1.\]
Thus, $(x_d)$ satisfies the conditions of Lemma \ref{2^a m}.

In order to obtain our desired result, we simply bound the number of possible tuples from below. By definition, each $x_d$ lies in the interval $[(n - 2^{a + 1} + 1)/(\sigma(m) - 1), n/(\sigma(m) - 1)]$, which contains at least $\floor{(2^{a + 1} - 1)/(\sigma(m) - 1)}$ integers. In addition, $\mathcal{D} (m) \backslash \{1\}$ contains $d(m) - 1$ elements. So, the total number of possible tuples is at least
\[\floor*{\frac{2^{a + 1} - 1}{\sigma(m) - 1}}^{d(m) - 1}. \qedhere\]
\end{proof}

\begin{proof}[Proof of Theorem \ref{ultimate}] This proof is similar to the previous one. We now let $(x_d)_{d \in \mathcal{S} \backslash \{1\}}$ be a tuple of integers which lie in the interval
\[\left[\frac{n - \sigma(q_1 \cdots q_k)}{\sigma(\textrm{rad}(m)) - 1}, \frac{n}{\sigma(\textrm{rad}(m)) - 1}\right],\]
where $\mathcal{S} = \mathcal{D} (\textrm{rad} (m))$.
In this case, we need to show that $(x_d)$ satisfies the conditions of Lemma \ref{squarefree lemma}.

We now have
\[x_d \leq \frac{n}{\sigma(\textrm{rad} (m)) - 1} \leq \sigma(q_1 \cdots q_k)\]
for all $d \in \mathcal{S} \backslash \{1\}$. In addition,
\begin{eqnarray*}
\sum_{d \in \mathcal{S} \backslash \{1\}} dx_d & \leq & \left(\sum_{d \in \mathcal{S} \backslash \{1\}} d\right) \frac{n}{\sigma(\textrm{rad} (m)) - 1} = n, \\
\sum_{d \in \mathcal{S} \backslash \{1\}} dx_d & \geq & \left(\sum_{d \in \mathcal{S} \backslash \{1\}} d\right) \frac{n - \sigma(q_1 \cdots q_k)}{\sigma(\textrm{rad} (m)) - 1} = n - \sigma(q_1 \cdots q_k).
\end{eqnarray*}
To finish the proof, we observe that the interval
\[\left[\frac{n - \sigma(q_1 \cdots q_k)}{\sigma(\textrm{rad} (m)) - 1}, \frac{n}{\sigma(\textrm{rad} (m)) - 1}\right]\]
contains at least $\floor{\sigma(q_1 \cdots q_k)/(\sigma(\textrm{rad} (m)) - 1)}$ integers and that $\mathcal{S} \backslash \{1\}$ has $d(\textrm{rad} (m)) - 1$ elements. Therefore, there are at least
\[\floor*{\frac{\sigma(q_1 \cdots q_k)}{\sigma(\textrm{rad} (m)) - 1}}^{d(\textrm{rad}(m)) - 1}\]
acceptable tuples.
\end{proof}

\end{document}